\documentclass{amsart}

\usepackage[T1]{fontenc}
\usepackage{times}
\usepackage{amssymb,epsfig,verbatim,xypic}
\usepackage{amsaddr}%
\theoremstyle{plain}

\usepackage[T1]{fontenc}
\usepackage[utf8]{inputenc}
\usepackage[french,english]{babel}
\usepackage{amssymb}
\usepackage{amsthm}
\usepackage{graphics}
\usepackage{amsmath}
\usepackage{amstext}
\usepackage{multirow}
\usepackage{hyperref}
\usepackage{comment}
\usepackage{relsize}
\usepackage{tikz-cd}

\usepackage{graphicx}
\usepackage{amsmath}
\usepackage{amssymb}
\usepackage{amsthm}
\usepackage{amstext}
\usepackage{epstopdf}
\usepackage{mathrsfs}
\usepackage{amscd}
\usepackage{xypic}

\usepackage[all,cmtip]{xy}

\newtheorem{thm}{Theorem}[section]
\newtheorem{lemma}[thm]{Lemma}

\newtheorem{THM}{Theorem}

\newtheorem{COR}[THM]{Corollary}

\theoremstyle{remark}
\newtheorem{ex}[thm]{Example}
\newtheorem{rem}[thm]{Remark}

\newcommand{\mb}{\mathbb}
\newcommand{\mc}{\mathcal}
\newcommand{\R}{\mb R}
\newcommand{\C}{\mb C}
\newcommand{\Pj}{\mb P}
\newcommand{\Z}{\mb Z}
\newcommand{\Q}{\mb Q}

\newcommand{\inv}{^{-1}}

\newcommand{\rato}{\dashrightarrow}

\DeclareMathOperator{\re}{Re}

\DeclareMathOperator{\Aut}{Aut}
\DeclareMathOperator{\Bir}{Bir}

\DeclareMathOperator{\Pic}{Pic}
\DeclareMathOperator{\PGL}{PGL}
\DeclareMathOperator{\id}{id}

\numberwithin{equation}{section}

\addtocounter{section}{0}             
\numberwithin{equation}{section}       

\title[Application of $p$-adic integration to transformations preserving a fibration]{An application of $p$-adic integration to the dynamics of a birational transformation preserving a fibration }

\begin{document}

\date{}
\author{Federico Lo Bianco}
\address{InterDigital France,\\  845a Av. des Champs Blancs, \\35510 Cesson-Sévigné (France)}
\email{federico.lobianco@interdigital.com}

\maketitle

\begin{abstract}
Let $f\colon X \rato X$ be a birational transformation of a projective manifold $X$ whose Kodaira dimension $\kappa(X)$ is non-negative. We show that, if there exist a dominant rational map $\pi \colon X\rato B$ and a birational transformation $f_B\colon B\rato B$ which preserves a big line bundle $L\in \Pic(B)$ and such that $f_B\circ \pi=\pi\circ f$, then $f_B$ has finite order.\\
As a corollary we show that, for projective irreducible symplectic manifolds of type $K3^{[n]}$ or generalized Kummer, the first dynamical degree characterizes the birational transformations admitting a Zariski-dense orbit.
\end{abstract}

\section{Introduction}
\label{intro}

Let $f\colon X\rato X$ be a birational transformation of a complex projective manifold. A natural question when studying the dynamical properties of $f$ is the existence of an equivariant rational fibration, i.e. of a dominant rational map with connected fibres $\pi\colon X\rato B$ onto a projective manifold and of a birational transformation $f_B\colon B\rato B$ such that the following diagram commutes:
$$
\begin{tikzcd}
M \arrow[dashrightarrow]{r}{f} \arrow[dashrightarrow,swap]{d}{\pi} & M \arrow[dashrightarrow]{d}{\pi} \\
B \arrow[dashrightarrow]{r}{f_B}& B
\end{tikzcd}.
$$
The transformation $f$ is called \emph{imprimitive} (see \cite{MR3431659}) if there exists a non-trivial $f$-equivariant fibration (i.e. such that $0<\dim B<\dim X$), and primitive otherwise; the study of imprimitive birational transformations should intuitively be simpler than primitive ones, as their dynamics decomposes into smaller dimensional dynamical systems: the base and the fibres. The goal of the present paper is to study the action on the base induced by an imprimitive transformation. When $\pi$ is (birationally equivalent to) the canonical fibration of $X$, some finite index subgroup of $\Bir(X)$ acts as the identity on $B$; this is a consequence of the finiteness of the pluricanincal representations, see \cite{MR0506253}. Our main result is the following generalization:

\begin{THM}
\label{main result}
Let $X$ be a complex projective manifold and let $f\colon X\rato X$ be a birational transformation. Suppose that there exist a dominant rational map $\pi\colon X\rato B$ onto a projective manifold $B$ and a birational transformation $f_B\colon B \rato B$ such that $f_B\circ \pi=\pi\circ f$. Assume that 
\begin{enumerate}
\item the Kodaira dimension $\kappa (X)$ of $X$ is non-negative;
\item $f_B$ preserves a big line bundle $L$.
\end{enumerate}
Then $f_B$ has finite order.
\end{THM}

Following \cite{MR1756108}, we say that a birational transformation $g\colon Y \rato Y$ preserves a line bundle $L$ on $Y$ if there exists a resolution of indeterminacies 
$$
\begin{tikzcd}
\widetilde{Y} \arrow[rightarrow,swap]{d}{\eta} \arrow[rightarrow]{dr}{\widetilde g}  \\
Y \arrow[dashrightarrow]{r}{g}& Y
\end{tikzcd}.
$$
such that $\eta^* L= \widetilde g^* L$; when $g$ is an automorphism, or more generally a pseudo-automorphism (i.e. $g$ and $g\inv$ do not contract any hypersurface), one can define the pull-back of a line bundle through $g$ and show that $L$ is preserved by $g$ if and only if $g^*L=L$. A birational transformation $g$ which preserves a line bundle $L$ induces a linear automorphism
$$g^*\colon H^0(Y,L) \to H^0(Y,L)$$
defined by $g^*(s):=\eta_* \widetilde g^*(s)$; this definition makes sense because, by the projection formula, $\eta_*\eta^*L=L$.

\begin{rem}
The second assumption of Theorem \ref{main result} is automatically verified if $\pm K_B$ is big and $g$ is a pseudo-automorphism; remark however that, if $K_B$ is big, i.e. $B$ is of general type, then the group of birational transformations is finite, which implies the conclusion of the Theorem. \\
If $X$ is irreducible symplectic and $\pi\colon X\to B$ is a regular fibration onto a smooth projective manifold, then $K_B$ is ample, i.e. $B$ is Fano (see \cite[Corollary 1.3]{MR2130616}).
\end{rem}


By the same approach we also obtain an analogous result concerning groups of transformations. Recall that the group $\Aut(X)$ of automorphisms of $X$ can be naturally seen as a Zariski-open subset of the Hilbert scheme of subvarieties of $X\times X$, which endows it with a natural topology; we denote by $\Aut_0(X)$ the connected component of the identity of $\Aut(X)$ and by $\Bir(X)$ the group of birational transformations of $X$.

\begin{THM}
\label{main result groups}
Let $X$ be a complex projective manifold with $\kappa(X)\geq 0$ and let $G\subset \Bir(X)$ be a group of birational transformations of $X$.
Suppose that there exist a dominant rational map $\pi\colon X\rato B$ onto a projective manifold $B$ which is preserved by $G$ and let $\pi_*G$ denote the image of the natural morphism $\pi_* \colon G \to \Bir(B)$. Assume that 
\begin{enumerate}
\item the quotient $G/(G\cap \Aut_0(X))$ is finitely generated;
\item all elements of $\pi_*G$ preserve a big line bundle $L$.
\end{enumerate}
Then $\pi_*G$ is finite.
\end{THM}

\begin{proof}
The Kodaira-Iitaka fibration associated with some multiple $L^{\otimes h}$ of $L$ is birational onto its image and allows us to identify $\pi_*G$ with a subgroup of $\PGL H^0(B,L^{\otimes h})=\PGL_{N+1}(\C)$; see \textsection \ref{first reduction} for details. By Theorem \ref{main result}, $\pi_*G$ is torsion.

The connected component $\Aut_0(X)$ has a natural structure of connected algebraic group; therefore, by Chevalley's structure theorem (see for example \cite[Theorem 8.27]{MR3729270}), it is isomorphic to an extension of an abelian variety by a linear algebraic group. Since $\kappa(X)\geq 0$, linear algebraic groups have trivial action on $X$ by \cite[Theorem 14.1]{MR0506253}; this means that $\Aut_0(X)$ is an abelian variety.\\
Let $G_0:=G \cap \Aut_0(X)$; note that the topological closure $\overline G_0 \leq \Aut_0(X)$ still preserves the fibration $\pi$ and that the induced action on $B$ preserves $L$ (this is a consequence of the seesaw theorem, see \cite[Corollary 5.6]{MR2514037}). In particular  $\pi_*$ extends to $\overline G_0$ and, since $\overline G_0$ is compact, so is its image $\pi_*\overline G_0$.\\
Again by Theorem \ref{main result}, $\pi_*\overline G_0$ is torsion. Compact torsion subgroups of Lie groups are finite by \cite{MR0393342}, thus $\pi_*\overline G_0$ is finite; hence a fortiori so is $\pi_*G_0$.

Let $K\lhd G_0$ denote the kernel of the restriction $\pi_* \colon G_0 \to \PGL_{N+1}(\C)$; since $K$ has finite index in $G_0$, the group $G/K$ is finitely generated. The homomorphism $\pi_*$ factors through a natural homomorphism
$$\phi \colon G/K \to \PGL_{N+1}(\C),$$
so that $\pi_*G=\phi(G/K)$.

By \cite{MR0393342}, a torsion subgroup of a Lie group is virtually abelian; furthermore, by Schreier's lemma a finite index subgroup of a finitely generated group is finitely generated. Since it is easy to show that an abelian, finitely generated torsion subgroup of $\PGL_{N+1}(\C)$ is finite, we obtain that a finite index subgroup of $\phi(G/K)$ is finite. In particular  $\pi_*G=\phi(G/K)$ is finite, which concludes the proof.
\end{proof}

The following corollary has the advantage of requiring only numerical hypothesis on the action of $f_B$, instead of having to compute its action on the Picard group.

\begin{COR}
\label{cor:no ample line bundle}
Let $X$ be a projective manifold and let $f\colon X\rato X$ be a birational transformation. Suppose that there exist a dominant rational map $\pi\colon X\rato B$ onto a projective manifold $B$ and a birational transformation $f_B\colon B \rato B$ such that $f_B\circ \pi=\pi\circ f$. Assume that 
\begin{enumerate}
\item the Kodaira dimension $\kappa (X)$ is non-negative;
\item $\Pic^0(B)=0$;
\item the induced linear maps $(f_B^N)^*\colon H^*(B,\C)\to H^*(B,\C)$ have bounded norm as $N\to +\infty$.
\end{enumerate}
Then $f_B$ has finite order.
\end{COR}

\begin{rem}
The second and third assumptions of Corollary \ref{cor:no ample line bundle} are automatically satisfied if $\Pic^0(X)=0$ and the induced linear maps $(f^N)^*\colon H^*(X,\C)\to H^*(X,\C)$ have bounded norm.
\end{rem}

\begin{proof}
Since the induced linear maps $(f_B^N)^*\colon H^*(B,\C)\to H^*(B,\C)$ have bounded norm as $N\to +\infty$, by Weil's regularization theorem (see \cite[Theorem 3]{MR1453608}), up to replacing $B$ by a smooth birational model and $f_B$ by an iterate, we may assume that  $f_B$ is an automorphism and that $f_B \in \Aut^0(B)$. In particular, $f_B$ has trivial action on $H^*(B,\C)$ and thus, since line bundles on $B$ are uniquely determined by their numerical class, on $\Pic(B)$. Therefore $f_B$ is an automorphism which preserves an ample line bundle, hence by Theorem \ref{main result} it has finite order.
\end{proof}

An almost equivalent formulation in terms of Kodaira-Iitaka fibrations is the following:

\begin{COR}
\label{cor:Kodaira-Iitaka}
Let $X$ be a complex projective manifold with non-negative Kodaira dimension and let $f\colon X\rato X$ be a birational transformation. If $f$ preserves a line bundle $\mc L$, then the induced projective automorphism
$$f^*\colon \Pj H^0(X,\mc L) \to \Pj H^0(X,\mc L)$$
has finite order. 
\end{COR}

As was pointed out to me by Vlad Lazi\'c, such a formulation is linked with the problem of determining the finiteness of pluri-log-canonical representations, which in turn plays a role in the problem of abundance conjecture. In this context one needs to consider the linear action of the group of birational transformations of a normal scheme $X$ preserving a divisor $\Delta$ on the space of sections $H^0(X,m(K_X+\Delta))$ (instead of its projectification as is done in the present work). Proving the finiteness of the linear action for all $m$ multiples of a certain $m_0$ (which is a stronger result than Corollary \ref{cor:Kodaira-Iitaka}), together with the application of the Minimal Model Program, allows to reduce the Abundance Conjecture to the case of log-canonical pairs (see \cite[Theorem 5.10]{MR1756108} and \cite[Theorem 1.4]{MR3200670}).  See \cite{MR1756108,MR3200670} for more details and for positive results in this direction.

\begin{rem}
Using Theorem \ref{main result groups}, it is not hard to extend Corollary \ref{cor:no ample line bundle} and \ref{cor:Kodaira-Iitaka} to groups $G$ of birational transformations such that $G/(G\cap \Aut_0(X))$ is finitely generated.
\end{rem}

\subsection{The case of irreducible symplectic manifolds}

\hyphenation{in-ter-es-ting}
Theorem \ref{main result} is particularly interesting in the case where $X$ is an irreducible symplectic manifold. A compact K\"ahler manifold is said to be \emph{irreducible symplectic} (or hyperk\"ahler) if it is simply connected and the vector space of holomorphic $2$-forms is spanned by a nowhere degenerate form. Irreducible symplectic manifolds form, together with Calabi-Yau manifolds and complex tori, one of the three fundamental classes of compact K\"ahler manifolds with trivial Chern class.

\begin{ex}
All $K3$ surfaces are irreducible symplectic; more generally, if $S$ is a $K3$ surface, then the Hilbert scheme $S^{[n]}$ of $n$ points on $S$ is an irreducible symplectic manifold of dimension $2n$.\\
Similarly, if $T$ is a two-dimensional complex torus and if
$$\Sigma \colon T^{[n+1]}\to T$$
denotes the sum morphism, i.e. $\Sigma[p_1,\ldots, p_{n+1}] := \sum _i p_i\in T$, then any fibre of $\Sigma$ is an irreducible symplectic manifold of dimension $2n$, called the generalized Kummer variety associated to $T$.
\end{ex}

We will say that $X$ is of type $K3^{[n]}$ (resp. generalized Kummer) if it is deformation equivalent to the Hilbert scheme of $n$ points on a $K3$ surface (resp. to a generalized Kummer variety); all known irreducible symplectic manifolds are of type $K3^{[n]}$, generalized Kummer or deformation equivalent to one of two sporadic examples by O'Grady. See \cite{MR1963559} for a complete introduction to irreducible symplectic manifolds.

Any birational transformation $f\colon X \rato X$ of an irreducible holomorphic symplectic manifold is actually a pseudo-automorphism; furthermore, $f$ induces a linear automorphism
$$f^*\colon H^2(X,\Z) \to H^2(X,\Z).$$
We will call the \emph{first dynamical degree} of $f$ the maximum $\lambda_1(f)$ of the moduli of eigenvalues of $f^*$; since $f^*$ preserves a closed salient cone, $\lambda_1(f)$ is actually a real positive eigenvalue of $f^*$.

The following result was proven in  \cite{MR3431659}.

\begin{thm}
\label{hu zhang} 
Let $X$ be a $2n$-dimensional projective irreducible symplectic manifold of type $K3^{[n]}$ or of type generalized Kummer and let $f\in \Bir(X)$ be a birational transformation with infinite order; the first dynamical degree $\lambda_1(f)$ is equal to $1$ if and only if there exist a rational Lagrangian fibration $\pi\colon X\rato \Pj^n$ and an automorphism $g\in Aut(\Pj^n)=PGL_{n+1}(\C)$ such that $\pi\circ f=g\circ \pi$.
\end{thm}

This allows to prove the following corollary:

\begin{COR}
Let $f$ be a birational transformation of a projective irreducible symplectic manifold $X$ of type $K3^{[n]}$ or generalized Kummer; then $f$ admits a Zariski-dense orbit if and only if the first dynamical degree $\lambda_1(f)$ is $>1$.
\end{COR}

\begin{proof}
If $\lambda_1(f)>1$, then by \cite[Main Theorem]{lobiancoprimitivity} the very general orbits of $f$ are Zariski-dense.

Assume conversely that $\lambda_1(f)=1$; by Theorem \ref{hu zhang} there exists a Lagrangian fibration $\pi\colon X\rato \Pj^n$ and a linear automorphism $g\in \Aut(\Pj^n)$ such that $\pi\circ f=g\circ \pi$. Since $g$ is biregular, it preserves the ample line bundle $\mc O_{\Pj^n}(1)$.\\
Therefore by Theorem \ref{main result} $g$ has finite order; in particular all the orbits of $f$ are contained in a finite union of fibres of $\pi$, hence they are not Zariski-dense.
\end{proof}

\section{Examples}
\begin{flushleft}
\begin{minipage}{0.65\textwidth}
\begin{ex}
Let $E=\C/\Lambda$ be an elliptic curve, $T=E\times E$ and
$$f_T\colon (x,y) \mapsto (x,x+y).$$
Then $f_T$ preserves the projection $\pi_1\colon X\to E$ on the first factor; the induced action on $E$ is the identity.\\
Let now
$$\tilde f_T\colon (x,y)\mapsto (x+a,x+y) \qquad \text{for some }a\in E.$$
Again $\tilde f_T$ preserves the projection $\pi_1$, but if $a$ is not a torsion element, then the induced action on $E$ is not finite. Such action preserves the numerical class of an ample line bundle, but not the line bundle itself; this example shows that assumption $(2)$ in Theorem \ref{main result} cannot be weakened to a numerical one.
\end{ex}
\end{minipage}
\begin{minipage}{0.3\textwidth}
    \centering
    \begin{tikzpicture}
\draw (0,0) rectangle (3.5,3.5);
\draw[thick,->] (0,0) -- (3.5,0) node[anchor= west] {E};
\draw[thick,->] (0,0) -- (0,3.5) node[anchor=north east] {E};
\foreach \x in {0.5,1,1.5,2,2.5,3}
	\draw[gray,very thin] (\x,\x) -- (\x,3.5);
\foreach \x in {0.5,1,1.5,2,2.5,3}
	\draw[red,  thick,->] (\x,0) -- (\x,\x);
\draw[red,  thick,->] (2,0) -- (2,2) node[anchor=south] {$\small{f_T}$};
\end{tikzpicture}
\end{minipage}
\end{flushleft}

\begin{ex}
Let $T$ be as above and let $S=Km(T)$ be the Kummer surface associated to $T$ (which is a $K3$ surface), i.e. $S$ is the minimal resolution of the singular quotient $T/\langle \pm \id\rangle$. The automorphism $f_T$ commutes with $-\id$, and thus it induces an automorphism $f_S\colon S\to S$. Similarly, the projection onto the first factor induces a fibration $\pi\colon S \to \Pj^1$, and one can easily see that the following diagram commutes:
\[
        \begin{tikzcd}[ampersand replacement=\&, column sep=small]
            S \ar[d, "\displaystyle{\pi}"'] \ar[r, "\displaystyle{f_S}"]  \&  S \ar[d, "\displaystyle{\pi}"', swap]\\
             \Pj^1 \ar[r, "\displaystyle{id}"'] \& \Pj^1
        \end{tikzcd}
    \]
This shows a non-trivial example of the situation described by Theorem \ref{main result} in dimension $2$.
\end{ex}

\begin{ex}
 Let $S$ be as above and let $X=S^{[n]}$ be the Hilbert scheme of $n$ points on $S$; $f_S$ and $\pi$ induce a commutative diagram \[
        \begin{tikzcd}[ampersand replacement=\&, column sep=small]
            X \ar[d, "\displaystyle{\pi}"'] \ar[r, "\displaystyle{f_S}"]  \&  X \ar[d, "\displaystyle{\pi}"', swap]\\
             \Pj^n \ar[r, "\displaystyle{id}"'] \& \Pj^n
        \end{tikzcd}
    \]
A similar construction can be carried out on the generalized Kummer varieties of $T$. This shows non-trivial examples of the situation described by Theorem \ref{main result} in any even dimension.
\end{ex}

\begin{ex}
 Let $X=Y\times \Pj^1$ where $Y$ is any projective manifold and let $f=(id_Y, g) \in \Aut(X)$, where $g\in \Aut(\Pj^1)$. Then the following diagram clearly commutes 
  \[
        \begin{tikzcd}[ampersand replacement=\&, column sep=small]
            X \ar[d, "\displaystyle{\pi_2}"'] \ar[r, "\displaystyle{f_S}"]  \&  X \ar[d, "\displaystyle{\pi_2}"', swap]\\
             \Pj^1 \ar[r, "\displaystyle{g}"'] \& \Pj^1
        \end{tikzcd}
    \]
The automorphism $g$ preserves the ample (hence big) line bundle $\mc O_{\Pj^1}(1)$, and yet it is generally not of finite order. This shows that the assumption on the Kodaira dimension of $X$ in Theorem \ref{main result} is necessary.
\end{ex}

\section{Elements of $p$-adic integration}

In this section we give an introduction to $p$-adic integration; see \cite{chambert2014motivic}, \cite[Chapter 3]{popa2011modern} and \cite{MR1743467}.

\subsection{$p$-adic and local fields}

We remind that, for a prime number $p$, the $p$-adic norm on $\Q$ is defined as
$$\left| p^n\cdot \frac ab\right|=p^{-n}\qquad p\nmid a,\quad p\nmid b.$$
We denote by $\Q_p$ the metric completion of $(\Q, |\cdot|_p)$; every element of $\Q_p$ can be uniquely written as a Laurent series
$$a=\sum _{n=n_0}^{+\infty}a_np^n\qquad a_i\in \{0,1,\ldots,p-1\}.$$
Denote by $\Z_p$ the closed unit ball in $\Q_p$. It is an integrally closed local subring of $\Q_p$ with maximal ideal $p\Z_p$ and residue field $\mathbb F_p$; its field of fractions is $\Q_p$, and it is a compact, closed and open subset of $\Q_p$.

A \emph{$p$-adic field} is a finite extension $K$ of $\Q_p$ for some prime $p$; on $K$ there exists a unique absolute value $|\cdot|_K$ extending $|\cdot|_p$. We denote by $\mathcal O_K$ the closed unit ball in $K$.

A \emph{local field} is a field $K$ with a valuation $|\cdot|\colon K\to \R_{\geq 0}$ such that $K$ with the induced topology is locally compact. We say that two local fields $(K,|\cdot |_K)$ and $(K',|\cdot |_{K'})$ are equivalent if there exists a ring isomorphism $\phi\colon K \overset{\sim}{\to} K'$ such that there exists $a\in \R^+$ satisfying $|\phi(x)|_{K'}=|x|_K^a$ for all $x\in K$.

\begin{thm}
A local field of characteristic $0$ is equivalent either to $\R$ or $\C$ endowed with the usual absolute values (archimedean case) or to a $p$-adic field endowed with the unique extension of $|\cdot|_p$ (non-archimedean case).
\end{thm}

\subsection{Measure on $K$}

On a locally compact topological group $G$ there exists a measure $\mu$, unique up to scalar multiplication, called the \emph{Haar measure of $G$} such that:
\begin{itemize}
\item any continuous function $f\colon G\to \C$ with compact support is $\mu$-integrable;
\item $\mu$ is $G$-invariant to the left.
\end{itemize}
Other important properties of the Haar measure are as follows: every Borel subset of $G$ is measurable; $\mu(A)>0$ for every nonempty open subset of $G$.

We consider $G=(\Q_p,+)$, and take on it the Haar measure $\mu$ normalized so that
$$\mu(\Z_p)=1.$$

\begin{ex}
It is easy to show that for every $m\geq 0$ one has $\mu(p^m\Z_p)=p^{-m}$.
\end{ex}

More generally, on a $p$-adic field $K$ we consider the Haar measure $\mu$ such that
$$\mu(\mathcal O_K)=1.$$

\subsection{Integration on $K$-analytic manifolds}
\label{integration}

Let $K$ be a $p$-adic field with norm $|\cdot|$. For any open subset $U\subset K^n$, a function $f\colon U\to K$ is said to be $K$-analytic if locally around each point it is given by a convergent power series. Similarly, we call $f=(f_1,\ldots,f_m)\colon U\to K^m$ a $K$-analytic map if all the $f_i$ are analytic.

As in the real and complex context, we define a $K$-analytic manifold of dimension $n$ as a Hausdorff topological space locally modelled on open subsets of $K^n$ and with $K$-analytic change of charts.\\
\begin{ex}\begin{enumerate}
\item Every open subset $U\subset K^n$ is a $K$-analytic manifold of dimension $n$; in particular, the set $\mathcal O_K^n\subset K^n$ is a $K$-analytic manifold.
\item The projective space $\Pj_K^n$ over $K$ is a $K$-analytic manifold.
\item Every smooth algebraic variety over $K$ is a $K$-analytic manifold; in order to see this one needs a $K$-analytic version of the implicit function theorem (see \cite[\textsection 1.6.4]{chambert2014motivic}).
\end{enumerate}
\end{ex}

Differential forms are defined in the usual way via charts: on a chart $U$ with coordinates $x_1,\ldots ,x_n$, a differential form of degree $k$ can be written as
$$\alpha=\sum_{|I|=k}f_I(x_1,\ldots, x_n)dx_{i_1}\wedge\ldots \wedge dx_{i_k}$$
with $f_I\colon U\to K$ functions on $U$; if the $f_I$ are $K$-analytic we say that the form is analytic.\\
Now take a maximal degree analytic differential form $\omega$; let $\phi\colon U\to K^n$ be a local chart, defining local coordinates $x_1,\ldots ,x_n$. In these coordinates we can write
$$\phi_* \omega=f(x_1,\ldots ,x_n)dx_1\wedge \ldots \wedge dx_n.$$
Then one can define a Borel measure $|\omega|$ on $U$ as follows: for any open subset $A\subset U$, we set
$$|\omega|(A)=\int_{\phi(A)}\left|f(x)\right|_K d\mu,$$
where $\mu$ is the usual normalized Haar measure on $\phi(U)\subset K^n$.\\
Similarly, let $\omega$ be a maximal degree pluri-form, i.e. a section of the analytic sheaf $(\Omega_X^n)^{\otimes m}$ for some $m>0$; let $\phi\colon U\to K^n$ be a local chart, defining local coordinates $x_1,\ldots ,x_n$. In these coordinates we can write
$$\phi_* \omega=f(x_1,\ldots ,x_n)(dx_1\wedge \ldots \wedge dx_n)^{\otimes m}.$$
Then one can define a Borel measure $\sqrt[m]{|\omega|}$ on $U$ as follows: for any open subset $A\subset U$, we set
$$\sqrt[m]{|\omega|}(A)=\int_{\phi(A)}\sqrt[m]{\left|f(x)\right|_K} d\mu,$$
where $\mu$ is the usual normalized Haar measure on $\phi(U)\subset K^n$.

Now let $\omega$ be a global section of $\Omega_X^n$ (resp. $(\Omega_X^n)^{\otimes m}$). To define a Borel measure $|\omega|$ (resp. $\sqrt[m]{|\omega|}$) on the whole manifold $X$, one uses partitions of unity exactly as in the real case. The only thing to check is that $|\omega|$ (resp. $\sqrt[m]{|\omega|}$) transforms precisely like differential forms when changing coordinates; this is a consequence of the following $K$-analytic version of the change of variables formula.

\begin{thm}[Change of variables formula]
Let $U$ be an open subset of $K^n$ and let $\phi\colon U\to K^n$ be an injective $K$-analytic map whose Jacobian $J_\phi$ is invertible on $U$. Then for every measurable positive (resp. integrable) function $f\colon \phi(U)\to \R$
$$\int_{\phi(U)}f(y)d\mu(y)=\int_U f(\phi(x))\left|\det J_{\phi}(x)\right|_K d\mu(x).$$
\end{thm}

\section{Proof of Theorem \ref{main result}}

In this section we give the proof of Theorem \ref{main result}. The strategy of the proof goes as follows:
\begin{enumerate}
\item A multiple of $L$ induces a rational map $\phi\colon B\rato \Pj^N$ which is birational onto its image; since $f_B$ preserves $L$, it induces a linear automorphism $g\in \Aut(\Pj^N(\C))=\PGL_{N+1}(\C)$ which preserves $\overline{\phi(B)}$.
\item Find an $f$-invariant volume form $\omega$ on $X$ (for $X$ irreducible symplectic, $\omega=(\sigma \wedge \bar \sigma)^n$,  where $2n=\dim X$ and $\sigma$ is a symplectic form).
\item The push-forward of $\omega$ by $\pi$ defines a $f_B$-invariant measure $\mu$ on $B$ not charging positive codimensional subvarieties; using this it is not hard to put \linebreak $g\in \PGL_{N+1}(\C)$ in diagonal form with only complex numbers of modulus $1$ on the diagonal.
\item Define the field of coefficients $k$: roughly speaking, a finitely generated (but not necessarily finite) extension of $\Q$ over which $X$, $B$, $f$, the volume form and all the relevant maps are defined.
\item Apply a key lemma: if one of the coefficients $\alpha$ of $g$ weren't a root of unity, there would exist an embedding $k\hookrightarrow K$ into a local field $K$ such that $|\rho(\alpha)|\neq 1$. Then the same measure-theoretic argument as in point $(3)$ leads to a contradiction.\\
A similar idea appears in the proof of Tits alternative for linear groups, see \cite{MR0286898}.
\end{enumerate}

\subsection{Invariant volume form on $X$}

Given a holomorphic $n$-form $\Omega$ ($n$ being the dimension of $X$), the pull-back $f^*\Omega$ is defined outside the indeterminacy locus of $f$; the latter being of codimension $\geq 2$, by Hartogs principle we can extend $f^*\Omega$ to an $n$-form on the whole $X$. This action determines a linear automorphism
$$f^*\colon H^0(X,K_X)\to H^0(X,K_X).$$
Similarly, for all $m\geq 0$ one can define linear automorphisms
$$f_m^*\colon H^0(X,mK_X)\to H^0(X,mK_X).$$
Since $X$ has non-negative Kodaira dimension, there exists $m>0$ such that $mK_X$ has a non-trivial section $\Omega$. By the finiteness of the pluricanonical representation (see \cite{MR0506253}), some finite index subgroup of the group $\Bir(X)$ of birational transformations of $X$ has trivial action on $H^0(X,mK_X)$; in particular, up to replacing $f$ by one of its iterates we may suppose that $f^*_m$ is the identity, so that in particular
$$f_m^*\Omega=\Omega.$$

The section $\Omega$ can be written in local holomorphic coordinates $x_1,\ldots ,x_n$ as
$$\Omega=a(x)(dx_1\wedge \ldots \wedge dx_n)^{\otimes m}$$
for some (local) holomorphic function $a$. 
Thus locally
$$\Omega \wedge \overline \Omega = \left| a(x)\right|^2 (dx_1\wedge \ldots \wedge dx_n)^{\otimes m} \wedge (d\bar x_1 \wedge \ldots \wedge d\bar x_n)^{\otimes m}.$$
It can be checked that the local form
$$\omega=\frac {(-1)^{n(n-1)/2}}{i^n}\sqrt[m]{\left| a(x)\right | ^2} dx_1\wedge \ldots \wedge dx_n \wedge d\bar x_1\wedge \ldots \wedge d\bar x_n$$
is a volume form outside the zeros of $\Omega$; since it is canonically associated to $\Omega$, such local expressions glue together to define a volume form on $X$
$$\omega=\frac {(-1)^{n(n-1)/2}}{i^n} \sqrt[m]{\Omega \wedge \overline \Omega}.$$

The push-forward by $\pi$ induces a measure $\mu$ on $B$: for all Borel set $A\subset B$, we set
$$\mu (A):=\int_{\pi^{-1}(A)}\omega;$$
here, if $U\subset X$ is a Zariski-open subset where $\pi$ is well defined, we denote by $\pi\inv(A)$ the set $\pi|_U\inv (A\cap \pi(U))\subset U$. Since $\omega$ is a volume form (and therefore doesn't charge positive codimensional subvarieties), the definition of $\mu$ is independent on the choice of the Zariski-open subset $U$; furthermore, $\mu$ doesn't charge positive codimensional subvarieties. The measure $\mu$ is $f_B$-invariant:
$$(f_B)_*\mu (A)=\mu (f_B\inv A)=\omega(\pi\inv f_B\inv A)=\omega(f\inv \pi\inv A)=\omega(\pi\inv A)=\mu(A).$$

\subsection{A first reduction of $f_B$}
\label{first reduction}

From now on we fix a multiple $L^{\otimes h}$ of $L$ such that the induced Kodaira-Iitaka map
$$\phi=\Phi_{hL}\colon B\rato \Pj H^0(B,L^{\otimes h})^\vee \cong \Pj^N$$
is birational onto its image; let $B_0=\overline {\phi(B)}$ and let $g\in \Aut(\Pj^N)=\PGL_{N+1}(\C)$ be the linear automorphism induced by $f_B$. Then $B_0$ is a (possibly singular) $g$-invariant variety and the measure $\mu$ on $B$ induces by push-forward a $g$-invariant measure on the image $\phi(U)$ of a Zariski-open subset $U\subset B$ where $\phi$ is we-defined; such measure can be extended to a $g$-invariant measure $\mu_0$ on $B_0$ by defining $\mu_0(S) = \phi_*\mu(S \cap \phi(U))$. Since $\mu$ doesn't charge positive codimensional subvarieties, neither does $\mu_0$.\\
In a given system of homogeneous coordinates on $\Pj^N$, an automorphism $g\in \Aut(\Pj^N)=\PGL_{N+1}(\C)$ is represented by a matrix $M$ acting linearly on such coordinates; $M$ is well-defined up to scalar multiplication. We will say that $g$ is semi-simple if $M$ is; in this case there exist homogeneous coordinates $Y_0,\ldots, Y_n$ such that the action of $g$ on these coordinates can be written
$$g([Y_0:\ldots :Y_N])=
\left[
\begin{array}{cccccccc}
1  \\
& \alpha_1 &\\
&&\ddots \\
&&& \alpha_N\end{array}
\right]
\underline Y=[Y_0:\alpha_1 Y_1:\ldots : \alpha_NY_N].$$
By an abuse of terminology, we will call the $\alpha_i$ the eigenvalues of $g$; because of the arbitrary choice of a homogeneous coordinate for which the corresponding diagonal element is equal to $1$, they are not well-defined. However, the property that they are all of modulus $1$ is: indeed, a different choice of a homogeneous coordinate, e.g. the $i$-th, simply divides the eigenvalues by $\alpha_i$.

\begin{lemma}
\label{lemma: first reduction}
The automorphism $g$ is semi-simple and its eigenvalues have all modulus $1$.
\end{lemma}
\begin{proof}
Let us prove first that $g$ is semi-simple. If this were not the case, the Jordan form of $g$ (which is well-defined up to scalar multiplication) would have a non-trivial Jordan block, say of dimension $k\geq 2$. We will consider the lower triangular Jordan form. In some good homogeneous coordinates $Y_0,\ldots Y_n$ of $\Pj^N$, after rescaling the coefficients of $g$ we can write
$$g(\underline Y)=
\left[
\begin{array}{cccccccc}
1  &&&\bf{0}\\
1 & 1 &\\
&\ddots&\ddots &&&\mathlarger{\mathlarger{\bf{0}}}\\
\bf{0}&&1& 1\\
&&&&\alpha_k&& \bf{0} \\
&\mathlarger{\mathlarger{\bf{0}}}&&&&\ddots\\
&&&&\bigstar&&\alpha_N
\end{array}
\right]
\underline Y.
$$
Take the affine chart $\{Y_0\neq 0\}\cong \C^N$ with the induced affine coordinates $y_i=Y_i/Y_0$. In these coordinates we can write
$$g(y_1,\ldots ,y_N)=(y_1+1,\ldots)$$
and thus
$$g^k(y_1,\ldots , y_N)=(y_1+k,\ldots).$$
\begin{minipage}{0.6\textwidth}
Let 
$$A=S\times \C^{N-1}, \qquad S=\{y_1\in \C \,|\, 0\leq \re(y_1)<1\};$$
 then we have
$$\C^N=\coprod_{k\in \Z} g^k(A).$$
\end{minipage}
\begin{minipage}{0.05\textwidth}\hfill\end{minipage}
\begin{minipage}{0.3\textwidth}
\begin{center}
\begin{tikzpicture}[x=0.5cm,y=0.5cm]
\fill[blue!20!white] (0,-3) rectangle (2,4);
\draw[very thick,->] (0,-3) -- (0,4) node[anchor= east] {Im};
\draw[very thick,->] (-3,0) -- (5,0) node[anchor=north] {Re};
\foreach \x in {-2,2,4}
	\draw[gray,thin] (\x,-3) -- (\x,4);
\draw[very thick, red, ->] (1,-2) -- (3,-2); 
\draw[very thick, red, ->] (1,2) -- (3,2) node[anchor=south east] {$g$};
\draw node [anchor=north, blue] at (1,0) {$S$};
\fill (0,0) circle [radius=2pt];
\fill (2,0) circle [radius=2pt];
\draw node [anchor=south east] at (0,0) {$0$};
\draw node [anchor=south west] at (2,0) {$1$};
\draw node [anchor=north] at (5,4) {$\C$};
\end{tikzpicture}
\end{center}\end{minipage}

Therefore
$$B_0\cap \C^N=\coprod_{k\in \Z} B_0\cap g^k(A)=\coprod_{k\in \Z} g^k(A\cap B_0).$$
We have $B_0 \cap \C^N \neq \emptyset$: indeed, if this were not the case, then there would exist a non-trivial section $s \in H^0(B,L^{\otimes h})$ such that $s(b) = 0$ for all $b\in B$, contradicting the non-triviality. Since $\mu_0$ doesn't charge positive-codimension subvarieties, we have 
\begin{eqnarray*}\mu(B)=\mu_0(B_0)=\mu_0 (B_0\cap \C^N)=\sum _{k\in \Z} \mu_0 (g^k(A\cap B_0))\\
=\sum _{k\in \Z} \mu_0 (A\cap B_0)=0 \text{ or }+\infty,
\end{eqnarray*}
which is a contradiction with the finiteness of $\mu$. This shows that $g$ is diagonalizable.

Next we show that, up to rescaling, in good homogeneous coordinates one can write
$$g(\underline Y)=
\left[
\begin{array}{cccccccc}
1  \\
& \alpha_1 &\\
&&\ddots \\
&&& \alpha_N\end{array}
\right]
\underline Y=[Y_0:\alpha_1 Y_1:\ldots : \alpha_NY_N]$$
with $|\alpha_i|=1$.\\
\begin{minipage}{0.6\textwidth}
Suppose by contradiction that $|\alpha_1|\neq 1$ (for example $|\alpha_1|>1$), and define 
$$A'=S'\times \C^{N-1} \qquad S'=\{ y_1\in \C\, |\, 1\leq |y_1|< |\alpha_1|\}.$$
\end{minipage}
\begin{minipage}{0.05\textwidth}\hfill\end{minipage}
\begin{minipage}{0.3\textwidth}
\begin{tikzpicture}[x=0.4cm,y=0.4cm]
\fill[fill=blue!20!white]
	(0,0) circle (2);
\fill[fill=white]
	(0,0) circle (1); 
\draw[very thick,->] (0,-4.5) -- (0,4.5) node[anchor= east] {Im};
\draw[very thick,->] (-4.5,0) -- (4.5,0) node[anchor=north] {Re};
\draw node [anchor=north] at (5,4) {$\C$};
\foreach \x in {.5,1,2,4}
	\draw (0,0) circle (\x);
\draw node [anchor=south west, blue] at (0,-2) {$S'$};
\draw[very thick, red, ->] (1,1) -- (2,2) node[anchor=south east] {$g$};
\draw[very thick, red, ->] (1,-1) -- (2,-2);
\draw[very thick, red, ->] (-1,1) -- (-2,2);
\draw[very thick, red, ->] (-1,-1) -- (-2,-2);
\fill (1,0) circle [radius=2pt];
\fill (2,0) circle [radius=2pt];
\draw node [anchor=south west] at (1,0) {$1$};
\draw node [anchor=south west] at (2,0) {$|\alpha|$};
\end{tikzpicture}
\end{minipage}
\\
The same argument as above leads to a contradiction.
\end{proof}

\subsection{The field of coefficients}

A key idea of the proof will be to define the "smallest" extension $k$ of $\Q$ over which $X$, $B$ and all the relevant applications are defined, and to embed $k$ in a local field in such a way as to obtain a contradiction.

Let us fix a cover of $X$ by affine charts $U_1,\ldots ,U_m$ trivializing the canonical bundle. Each of these $U_i$ is isomorphic to the zero locus of some polynomials $p_{i,1},\ldots, p_{i,n_i}$ in an affine space $\C^{N_i}$; fix some rational functions $g_{i,j}\colon \C^{N_i}\rato \C^{N_j}$ giving the changes of coordinates from $U_i$ to $U_j$.
 Let $f\colon \C^{N_i}\rato \C^{N_j}$ (resp. $\Omega_i\colon \C^{N_i}\rato \C$,) be some rational functions defining $f$ (resp. $\Omega$). Analogously, fix affine charts for $B$ trivializing $L$, local equations of $B$, change of coordinates for $L$ and rational functions locally defining $f_B\colon B\rato B$ and $\phi\colon B\rato \Pj^N$; here, we will fix homogeneous coordinates on $\Pj^N$ diagonalizing $g$ (see paragraph \ref{first reduction}). Finally, we fix equations of a resolution of indeterminacies of $f_B$ as follows
 $$
\begin{tikzcd}
\widetilde{B} \arrow[rightarrow,swap]{d}{\eta} \arrow[rightarrow]{dr}{\widetilde f_B}  \\
B \arrow[dashrightarrow]{r}{f_B}& B
\end{tikzcd}.
$$
 such that $\eta^*L=\widetilde{f}_B^* L$ (such a resolution exists by the assumption that $f_B$ preserves $L$).\\
We define the \emph{field of coefficients} $k=k_\Omega$ as the extension of $\Q$ generated by all the coefficients appearing in the $p_{i,k},f_{i,j},g_{i,j},h_{i,j}, \Omega_i, \pi_i$, in the equations of $B$, $f_B$, $\phi$, $\widetilde B$, $\eta$ and $\widetilde f_B$, and by the coefficients $\alpha_1,\ldots ,\alpha_N$ of $g$; this is a finitely generated (but not necessarily finite) extension of $\Q$ over which $X$, $B$ and all the relevant functions are defined.

Let $\rho\colon k\hookrightarrow K$ be an embedding of $k$ into a local field $K$; since $\R$ is naturally embedded in $\C$, we may and will assume that $K$ is either $\C$ or a $p$-adic field. We can now apply a base change in the sense of algebraic geometry to recover a smooth projective scheme over $K$ and all the relevant functions.

Here are the details of the construction: the polynomials $p_{i,k}^\rho=\rho(p_{i,k})$ define affine varieties $X_i^\rho$ of $K^{N_i}$; the rational functions $g_{i,j}^\rho$ allow to glue the $X_i^\rho$-s into a complex algebraic variety $X^\rho$. This variety is actually smooth since smoothness is a local condition which is algebraic in the coefficients of the $p_{i,k}$. Furthermore, by applying $\rho$ to all the relevant rational functions, we can recover a birational transformation $f^\rho\colon X^\rho \rato X^\rho$ and a pluricanonical section $\Omega^\rho\in H^0(X^\rho, mK_{X^\rho})$. Note that we can suppose that $X^\rho$ is projective: indeed, $X\subset \Pj^M(\C)$ is the zero locus of some homogeneous polynomials $P_1,\ldots ,P_k\in \C[Y_0,\ldots,Y_M]$, and, up to adding the affine open subsets $X_i=X\cap \{Y_i\neq 0\}$ to the above constructions, it is easy to see that $X^\rho\subset \Pj^N(K)$ is the zero locus of $P_1^\rho, \ldots, P_k^\rho$.
Analogously, applying $\rho$ to all the relevant equations allows to define a (smooth projective) variety $B^\rho$, a dominant rational map $\pi^\rho\colon X^\rho\dashrightarrow B^\rho$ and a birational transformation $f_B^\rho \colon B^\rho \rato B^\rho$ such that $\pi^\rho\circ f^\rho=f_B^\rho\circ \pi^\rho$:\\
 \centerline{\xymatrix{
X^\rho \ar@{-->}[d]^{\pi^\rho} \ar@{-->}[r]^{f^\rho} & X^\rho \ar@{-->}[d]^{\pi^\rho}\\
B^\rho \ar@{-->}[r]^{f_B^\rho} & B^\rho
}}
Furthermore, we get a line bundle $L^\rho$ which is $f_B^\rho$-invariant (meaning that there exists a resolution of indeterminacies of $f^\rho$, namely $\widetilde B^\rho$, such that $\eta^{\rho*}L^\rho=\widetilde f_B^{\rho*}$), and a rational map $\phi^\rho\colon B^\rho \rato \Pj^N_K$, constructed as the Kodaira-Iitaka map of $hL^\rho$, which is birational onto its image; denoting by $g^\rho\colon \Pj^N_K\to \Pj^N_K$ the automorphism given by $g^\rho[Y_0:\ldots:Y_n]=[Y_0:\alpha_1^\rho Y_1:\ldots :\alpha_N^\rho Y_n]$, $g^\rho$ preserves $B_0^\rho=\overline{\phi^\rho(B^\rho)}$, and $g^\rho|_{B_0^\rho}$ identifies with $f_B^\rho$.

We will denote by $\omega^\rho$ the measure on $X^\rho$ associated to $\Omega^\rho$:  this has been denoted by $\sqrt[m]{|\Omega^\rho|}$ in section \ref{integration} in the non-archimedean case, while if $K=\C$ it is defined as the measure of integration of 
$$\frac {(-1)^{n(n-1)/2}}{i^n} \sqrt[m]{\Omega^\rho \wedge \overline{\Omega^\rho}}.$$
In both cases, $\omega^\rho$ doesn't charge positive codimensional analytic subvarieties.

\begin{rem}
\label{if algebraic integer}
At this stage we can already prove that the $\alpha_i$ are algebraic numbers all of whose conjugates over $\Q$ have modulus $1$. Indeed suppose that this is not the case, say for $\alpha_1$; by a standard argument in Galois theory (see for example \cite{MR1878556}), one can find an embedding $\rho \colon k\hookrightarrow \C$ such that $|\rho(\alpha_1)|\neq 1$. Now, $g^\rho$ preserves the measure $\mu^\rho$ induced on $\Pj^n$ by $\omega^\rho$
$$\mu^\rho(A):=\int_{(\pi^\rho)\inv (A)}\omega^\rho,$$
and Lemma \ref{lemma: first reduction} leads to a contradiction.

If we somehow knew that the $\alpha_i$ are algebraic integers, we could conclude by a lemma of Kronecker's (see \cite{MR1578994}) that they are roots of unity. However, this is in general not true for algebraic numbers: for example, 
$$\alpha=\frac {3+4i}5$$
has only $\bar\alpha$ as a conjugate over $\Q$, and they both have modulus $1$, but they are not roots of unity. In order to exclude this case we will have to use the $p$-adic argument.
\end{rem}

\subsection{Key lemma and conclusion}

In his original proof of the Tits alternative for linear groups \cite{MR0286898}, Tits proved and used (much like we do in this context) the following simple but crucial lemma:

\begin{lemma}[Key lemma]
\label{embedding}
Let $k$ be a finitely generated extension of $\Q$ and let $\alpha\in k$ be an element which is not a root of unity. Then there exist a local field $K$ (with norm $|\cdot |$) and an embedding $\rho \colon k \hookrightarrow K$ such that $|\rho(\alpha)|>1$.
\end{lemma}

\begin{proof}[Proof of Theorem \ref{main result}]
Suppose by contradiction that one of the eigenvalues, say $\alpha_1$, is not a root of unity and define the field of coefficients $k$. By Lemma \ref{embedding}, there exists a local field $K$ and an embedding $\rho \colon k \hookrightarrow K$ such that $|\rho(\alpha_1)|_K>1$.

Now, since $f^*\Omega=\Omega$, we have $(f^\rho)^*(\Omega^\rho)=\Omega^\rho$, and in particular $f_B^\rho$ preserves the measure $\mu^\rho$ on $B^\rho$ induced by the push-forward of $\omega_\rho$:
$$\mu^\rho(A):=\omega^\rho\left((\pi^\rho)\inv (A)\right).$$
The measure $\mu^\rho$ is non-trivial, finite, and doesn't charge positive codimensional analytic subvarieties of $B^\rho$, thus we can conclude just as in the proof of Lemma \ref{lemma: first reduction}.

Namely, let $\mu_0^\rho$ be the measure on $B_0^\rho$ induced by $\mu^\rho$, identify $\{Y_0\neq 0\}\subset \Pj_K^N$ with $K^N$ and let 
$$A:=\{(y_1,\ldots,y_N)\in K^N\,|\,1\leq |y_1|<|\alpha_1|\}\subset K^N;$$ 
 then, since $\mu^\rho$ doesn't charge positive codimension analytic subsets, we have
\begin{eqnarray*}\omega^\rho(X^\rho)=\mu^\rho(B^\rho)=\mu_0^\rho(B_0^\rho \cap K^n)\\
=\sum_{N\in \Z} \mu_0^\rho((g^\rho)^N(A\cap B_0^\rho))=\sum _{N\in \Z}\mu_0^\rho(A \cap B_0^\rho)=0\text{ or }+\infty,
\end{eqnarray*}
a contradiction.
\end{proof}

\bibliography{references}{}
\bibliographystyle{alpha}

\end{document}